\documentclass[11pt]{amsart}
\usepackage{graphicx}
\usepackage{amssymb, amsmath}
\newtheorem{theorem}{Theorem}[section]
\newtheorem{corollary}[theorem]{Corollary}

\newtheorem{proposition}[theorem]{Proposition}
\theoremstyle{definition}
\newtheorem{definition}[theorem]{Definition}
\newtheorem{example}[theorem]{Example}
\newtheorem{remark}[theorem]{Remark}
\theoremstyle{remark}

\begin{document}
\title{Topo-Groups and a Tychonoff Type Theorem}
\author{\sc M. Shahryari}
\thanks{}
\address{ Department of Pure Mathematics,  Faculty of Mathematical
Sciences, University of Tabriz, Tabriz, Iran }

\email{mshahryari@tabrizu.ac.ir}
\date{\today}

\begin{abstract}
We define a topo-system on a group $G$ as a set of subgroups satisfying certain topology-like conditions.
We study some fundamental properties of groups equipped with a topo-system (topo-groups) and using a suitable
formalization of a subgroup filter, we prove a Tychonoff type theorem for the direct product of topo-compact topo-groups.
\end{abstract}

\maketitle



\vspace{2cm}

\section{introduction}
In this article, we introduce an interesting topology-like concept
concerning groups (and with  almost the same method it can be
defined for other algebraic systems). Given an arbitrary group $G$,
we define a {\em topo-system} on $G$ as a set of subgroups
satisfying certain conditions like a topology on a set. We will call
such a group, a {\em topo-group}. These topo-groups are not rare and
as we will see, there are many examples of topo-groups. We
investigate fundamental notions concerning topo-groups and we see
that many basic concepts of topology can be introduced in the frame
of topo-groups. A {\em filter of subgroups} will be defined in such
way that we will be able to formulate a Tychonoff type theorem for
the direct product of {\em topo-compact} topo-groups. Through this
article, $Sub(G)$ and $Sub^{\ast}(G)$ will be used for the set of
subgroups and the set of non-identity subgroups of $G$,
respectively.

\begin{definition}
Let $G$ be a group and $T\subseteq Sub(G)$ be a set satisfying:

a) $1$ and $G\in T$.

b) if $\{ A_i\}_{i\in I}$ is a family of elements of $T$, then $\langle A_i: i\in I\rangle \in T$.

c) if $A$ and $B$ are elements of $T$, then $A\cap B\in T$.

Then we call $T$ a topo-system on $G$ and the pair $(G,T)$ a topo-group.
\end{definition}

Elements of $T$ will be called {\em $T$-open} (or {\em topen} for short) subgroups.

\begin{example} It is quite easy to find many examples of topo-groups. Here we give a sample list of such examples:

1- Let $T=Sub(G)$. This is the {\em discrete} topo-system on $G$.

2- Let $T=\{ 1, G\}$, the trivial topo-system.

3- Let $B\subseteq G$ and $T_B=\{ A\leq G: B\subseteq A\}\cup \{1\}$. This is the {\em principal} topo-system associated to $B$.

4- Let $T_{cf}=\{ A\leq G: [G:A]<\infty\}\cup\{1\}$. This is the {\em cofinite} topo-system on $G$.

5- Let $T_n=\{ A\leq G: A\unlhd G\}$. This is the {\em normal} topo-system on $G$.


6- Let $T_{char}=\{ A\leq G: A=characteristic\}$. This is the {\em characteristic} topo-system on $G$.

7- Let $G$ be a topological group and $T$ be the set of all open
subgroups of $G$ together with the identity subgroup. Then $T$ is a
topo-system on $G$.

8- Let $\mathfrak{X}$ be a variety of groups and $T_{\mathfrak{X}}=\{ A\unlhd G: G/A\in \mathfrak{X}\}\cup\{ 1\}$.
Then $T_{\mathfrak{X}}$ is a topo-system on $G$ and in fact we have $T_{\mathfrak{X}}=T_H\cap T_n$, where $H=O^{\mathfrak{X}}(G)$,
is the $\mathfrak{X}$-residual of $G$.
\end{example}

\begin{definition}
A subgroup $A\leq G$ is {\em $T$-closed}, if for all $x\in G\setminus A$, there exists $B\in T$ such that $x\in B$ and $A\cap B=1$.
\end{definition}

It is easy to verify that $1$ and $G$ are $T$-closed, and the intersection of any family of $T$-closed subgroups is $T$-closed.
Let $X\leq G$ and $x\in X$. Suppose there is a topen $A$ with $x\in A\leq X$. Then we say that $x$ is an {\em interior} element of $X$.
The set of all interior elements of $X$ is denoted by $X^{\circ}$ and it is called the interior of $X$. Clearly the interior of $X$ is a
topen subgroup. As an example, in the case of the normal topo-system (see the above example), we have $X^{\circ}=X_G$, the core of $X$.
In the case of the cofinite topo-system, it can be verified that $X^{\circ}=1$ or $X$. For any subgroup $X$, the {\em boundary} is
$\partial X=X\setminus X^{\circ}$. A point $x\in G$ is a limit point of $X$, if for all topen $A$ containing $x$, we have $|A\cap X|\geq 2$.
The $T$-{\em closure} of $X$ is $\overline{X}$ which is the subgroup generated by $X$ and its limit points. Clearly $\overline{X}$ is $T$-closed.
Despite, ordinary topology, there is no symmetry between the concepts of topen and $T$-closed subgroups. For example, if any subgroup of a
topo-group $G$ is topen, we can not say that any subgroup is $T$-closed as well. For example, suppose the group $G=\mathbb{Z}\times \mathbb{Z}_4$
with the discrete topo-system. Then the subgroup $A=\langle (0,2)\rangle$ is not $T$-closed. In general, we have the next proposition.

\begin{proposition}
Let all cyclic subgroups of $G$ be $T$-closed. Then all non-identity elements of $G$ have prime order.
\end{proposition}

\begin{proof}
Let $x\neq 1$. Then $A=\langle x^2\rangle$ is $T$-closed. If $x$ does not belong to $A$, then there is a $B\in T$ containing $x$
such that $A\cap B=1$. So, $\langle x\rangle\cap A=1$ and hence $x^2=1$. If $x\in A$, then $x$ has finite order $n$.
Suppose $n=ab$. Then $\langle x^a\rangle$ is $T$-closed and hence
$$
x\in \langle x^a\rangle\ or\ \langle x^a\rangle\cap \langle x\rangle =1.
$$
In the first case we have $x^b=1$ and in the second case, we have $x^a=1$. So $n$ is a prime.
\end{proof}

\begin{definition}
A subgroup $X\leq G$ is {\em topo-compact} subgroup, if any topen
covering of $X$ has a finite sub-covering, i.e. if $X\subseteq
\cup_{i\in I}A_i$, with $A_i\in T$, then there is a finite subset
$J\subseteq I$, such that $X\subseteq \cup_{i\in J}A_i$. In the
special case $X=G$, we say that $G$ is topo-compact.
\end{definition}

For example a group is topo-compact with respect to the discrete
topo-system, iff it is a finite union of cyclic subgroups. In fact
such a group is finite or infinite cyclic as the next proposition
shows.

\begin{proposition}
Let $(G, T)$ be topo-compact and suppose that $T$ is the discrete topo-system. Then $G$ is finite or infinite cyclic.
\end{proposition}

\begin{proof}
We have $G=\cup_x\langle x\rangle$ and all $\langle x\rangle $ are
topen. So by topo-compactness, $G=\cup_{i=1}^n\langle x_i\rangle$.
Now by a theorem of B. H. Neumann (see \cite{Pass}, page 120), we
may assume that any $\langle x_i\rangle$ has  finite index. Suppose
$G$ is infinite. Then all $\langle x_i\rangle$ are infinite and
hence $G$ is torsion free. Assume that  $1\neq H\leq G$, and $1\neq
h\in H$. Then there is an index $i$ such that $h\in \langle
x_i\rangle$ and therefore $[\langle x_i\rangle: \langle h\rangle]$
is finite. So already $[G: \langle h\rangle]$ is finite and hence
$H$ has finite index. A theorem of Fedorov (see \cite{Scott}, page
446) says that if any non-trivial subgroup of an infinite group has
finite index, then the group is isomorphic to $\mathbb{Z}$. This
completes the proof.
\end{proof}

The concept of topo-compact groups may have many interesting
interpretations if we consider different topo-systems. It can be
verified that a $T$-closed subgroup of a topo-compact group is a
topo-compact subgroup.

If $(G, T)$ is a topo-group and $H\leq G$, then the set $S=\{ A\cap
H: \ A\in T\}$ generates a topo-system on $H$, which is {\em
induced} topo-system on $H$. We denote this topo-system  by
$T_{ind}$.  We give here a recursive construction of the elements of
$T_{ind}$. Let $T_0=S$ and for any ordinal $\alpha$, $$
T_{\alpha+1}=\{ \bigcap_{i=1}^n\langle B_{ij}: j\in I_i\rangle:\
n\geq 1, \{B_{ij}\}_{j\in I_i}\subseteq T_{\alpha}\}.
$$
Also for  limit ordinals we set
$$
T_{\alpha}=\bigcup_{\beta<\alpha}T_{\beta}.
$$
Then it can be shown that $T_{ind}=\cup_{\alpha}T_{\alpha}$. Now if
$(H, T_{ind})$ is  topo-compact topo-group, then it can be proved
that $H$ is a topo-compact subgroup of $G$. The converse situation
is very complicated because the topen subgroups of $(H, T_{ind})$
have very complex structures.

Also if $H\unlhd G$, then one checks that the set
$$
T^{\prime}=\{ \frac{AH}{H}:\ A\in T\}
$$
is a topo-system on the quotient group $G/H$. So, we call $G/H$, the
quotient topo-system. If $\{ (G_i, T_i)_{i\in I}\}$ is a family of
topo-groups and $G=\prod_{i\in I}G_i$, then the set
$$
T=\{ \prod_{i\in I}A_i:\ A_i\in T_i\ and\ almost\ all\ A_i=G_i\}
$$
is a topo-system on $G$. Note that the situation here is much better
than the case of ordinary product topology because of the following
two trivial equalities:

1- $(\prod_{i\in I}A_i)\cap(\prod_{i\in I}B_i)=\prod_{i\in I}(A_i\cap B_i)$.

2- $\langle \prod_{i\in I}A_{ij}: j\in J\rangle=\prod_{i\in I}\langle A_{ij}: j\in J\rangle$.

\begin{definition}
Let $(G, T)$ and $(H, S)$ be two topo-groups. A continuous map (or a
{\em topomorphism}) is any homomorphism $f: G\to H$ with the
property $f^{-1}(B)\in T$, for all $B\in S$. An invertible
topomorphism with a continuous inverse is a {\em topeomorphism}.
\end{definition}

Clearly the projections $\pi_j:\prod_{i\in I}G_i\to G_j$ are
topomorphisms and also the natural map $q:G\to G/H$ is also a
topomorphism. If $H$ is a subgroup of a topo-group $(G, T)$, then
the inclusion map $j:(H, T_{ind})\to (G, T)$ is a topomorphism, too.

\begin{remark}
Let $(G, T)$ be a topo-group. We can define a topology on $G$ with
the basis $T$. Let us denote this topology by $T^{\ast}$. A subset
$X\subseteq G$ is open, if $X$ is a union of elements of $T$. This
topology is never Hausdorff but $G$ is compact if and only if $G$ is
topo-compact. For a subgroup $H\leq G$ with the induced topo-system
$T_{ind}$, we have $(T^{\ast})_{ind}\subseteq (T_{ind})^{\ast}$.
\end{remark}

We close  this section, by two more examples of interesting topo-systems on arbitrary groups.

\begin{example}
Let $G$ be a group and $H\leq K\leq G$. We define
$$
T^{H,K}=\{ A\leq G:\ [A,K]\subseteq H\}\cup \{ G\}.
$$
Clearly $1$ and $G$ belong to $T^{H,K}$. This set is closed under
intersection, so let $\{ A_i\}_{i \in I}$ be a family of elements of
$T^{H,K}$. Suppose $a_1, \ldots, a_n$ are elements from the union of
$A_i$'s and $x\in K$. We have
$$
[a_1a_2\ldots a_n, x]=a_1\ldots a_n x a_n^{-1}\ldots a_1^{-1}x^{-1}.
$$
Now, for any $u\in K$, we have $a_iua_i^{-1}u^{-1}\in H$, so for some $h_n\in H$, we have
$$
[a_1a_2\ldots a_n, x]=a_1\ldots a_{n-1}(h_nx)a_{n-1}^{-1}\ldots a_1^{-1}x^{-1}.
$$
Continuing this way, for some $h_{n-1}\in H$, we have
$$
[a_1a_2\ldots a_n, x]=a_1\ldots a_{n-2}(h_{n-1}h_nx)a_{n-2}^{-1}\ldots a_1^{-1}x^{-1}.
$$
Hence, finally we get
$$
[a_1a_2\ldots a_n, x]=h_1h_2\ldots h_n,
$$
for some $h_1, \ldots, h_n\in H$. This shows that $\langle A_i: i\in I\rangle \in T^{H,K}$.
\end{example}

\begin{example}
Let $G$ be group and $H\leq G$ be a fixed subgroup. Define
$$
T_{conj(H)}=\{ A\leq G: H\subseteq N_G(A)\}.
$$
It can be verified that $T_{conj(H)}$ is a topo-system on $G$.
\end{example}

\section{Hausdorff topo-groups}
Suppose $(G, T)$ is a topo-group and for any $x, y\in G$ with $\langle x\rangle \cap\langle y\rangle =1$, there exist $A, B\in T$ with the properties
$$
x\in A,\ y\in B,\ and\ A\cap B=1.
$$
Then we say that $(G, T)$ is a Hausdorff topo-system. A subgroup
$T\leq G$ is {\em weak $T$-closed}, if for all $x$ with $\langle
x\rangle \cap A=1$, there exists $B\in T$ such that $A\cap B=1$.

\begin{proposition}
Let $(G, T)$ be a Hausdorff topo-group and $A$ be a topo-compact subgroup. Then $A$ is weak $T$-closed.
\end{proposition}

Clearly any discrete topo-group is Hausdorff. If $G$ is infinite and
$(G, T_{cf})$ is Hausdorff, then for all non-identity $x, y\in G$,
we have
$$
\langle x\rangle \cap\langle y\rangle\neq 1,
$$
so the intersection of any two non-trivial subgroups of $G$ is
non-trivial. Clearly such a group is torsion free or a $p$-group for
some prime. The groups $\mathbb{Z}$, $\mathbb{Q}$ and
$\mathbb{Z}_{p^{\infty}}$ are examples of such groups. Torsion free
non-abelian groups with this property are  constructed by Adian and
Olshanskii, \cite{Olshan}. An elementary argument shows that if $G$
is an R-group (i.e. a group satisfying $x^m=y^m\Rightarrow x=y$ for
any non-zero $m$), then $G$ can be embedded in $\mathbb{Q}$. The
next proposition shows that the $p$-group case can be investigate by
a very elementary argument, and so the hardest part of the problem,
is the case of torsion free groups which are not R-groups.

\begin{proposition}
Let $G$ be an infinite $p$-group which is Hausdorff with respect to
the topo-system $T_{cf}$. Then $G$ has a unique subgroup of order
$p$. The converse is also true.
\end{proposition}

\begin{proof}
Let $G$ be a $p$-group and the intersection of any two non-trivial
subgroup of $G$ be non-trivial. Suppose $x$ is a non-identity
element of $G$. Then $\langle x\rangle$ has a subgroup $A$ of order
$p$. For any $1\neq B\leq G$, we have $A\cap B\neq 1$, so
$A\subseteq B$ and hence $A$ is a unique subgroup of $G$ of order
$p$. Conversely let $G$ has a unique subgroup of order $p$, say $A$.
Then clearly for any non-trivial subgroup  $B$, we have $A\subseteq
B$, and hence every two non-trivial subgroups of $G$ have
non-trivial intersection.
\end{proof}

Note that the subgroup $A$ in the proof the above proposition is
{\em minimum} in the set of non-trivial normal subgroups of $G$. So,
by a well-known theorem of Birkhoff, such a group $G$ is
sub-directly irreducible.

\begin{corollary}
Any  infinite $p$-group which is Hausdorff with respect to the topo-system $T_{cf}$ is sub-directly irreducible.
\end{corollary}

For any group $G$, the topo-group $(G, T_n)$ is Hausdorff, if and only if $G$ satisfies
$$
\langle x\rangle \cap\langle y\rangle=1\Rightarrow \langle x^G\rangle \cap\langle y^G\rangle=1.
$$
In the next theorem, we give some properties of this kind of groups.
Note that a $\mathbb{Q}$-group is a group in which every element $x$
has a unique $m$-th root $x^{1/m}$, for every $m$.

\begin{theorem}
Let $G$ be a torsion free group which is Hausdorff with respect to
the topo-system $T_n$. Then for any $x, u\in G$, there are non-zero
integers $m$ and $n$ such that $ux^nu^{-1}=x^m$. Further, if $G$ is
also a $\mathbb{Q}$-group, then there exists a  family of normal
subgroups $A_i$, such that

1- $G= \bigcup_i A_i$.

2- $A_i\cong \mathbb{Q}$.

3- $A_i\cap A_j=1$, for $i\neq j$.

As a result, $G$ is abelian.
\end{theorem}

\begin{proof}
Let $A=G\setminus 1$ and define a binary relation on $A$ by
$$
x\equiv y \Leftrightarrow \langle x\rangle \cap\langle y\rangle \neq 1.
$$
This is an equivalence relation on $A$. Let $E(x)$ be the
equivalence class of $x$ and $\{ E_i=E(x_i): i\in I\}$ the set of
all such classes. Then we have
$$
G=\bigcup_i \bigcup_{y\in E_i}\langle y\rangle.
$$
Now, since
$$
\langle x\rangle \cap\langle y\rangle =1\Rightarrow  \langle x^G\rangle \cap\langle y\rangle=1,
$$
so, $\langle x^G\rangle\subseteq \bigcup_{y\in E(x)}\langle
y\rangle$. Let $u\in G$ be arbitrary. Then there is a $y\in E(x)$
such that $uxu^{-1}\in \langle y\rangle$. Hence, for some $i, m$ and
$n$ we have
$$
uxu^{-1}=y^i, \ \ x^m=y^n.
$$
Therefore, for any $x, u\in G$, there are non-zero integers $m$ and
$n$ such that $ux^nu^{-1}=x^m$. Now, suppose $G$ is a
$\mathbb{Q}$-group. Then we have
\begin{eqnarray*}
E_i&=&\{ y\in G:\ \exists m, n\neq 0, \ x_i^m=y^n\}\\
   &\subseteq& \{ x_i^{\alpha}:\ \alpha\in\mathbb{Q}\}\\
   &\subseteq& G.
\end{eqnarray*}

Suppose $A_i=\{ x_i^{\alpha}:\ \alpha\in\mathbb{Q}\}$. Then any
$A_i$ is a subgroup of $G$ and clearly, we have 1-2-3. Note that,
since for any $u\in G$ we have $ux_iu^{-1}=x_i^{\alpha}$, for some
$\alpha\in \mathbb{Q}$, so $A_i$ is a normal subgroup. Now, $[A_i,
A_j]\subseteq A_i\cap A_j$, so by 3, we have $[A_i, A_j]=1$. This
shows that $G$ is abelian.
\end{proof}

We also can use the Malcev completion of torsion free locally
nilpotent groups, to prove the next result on Hausdorff groups with
respect to $T_n$.

\begin{theorem}
Let $G$ be a torsion free locally nilpotent group which is Hausdorff
with respect to $T_n$. Then $G$ is abelian.
\end{theorem}

\begin{proof}
The same argument as above, shows that for any $x, u\in G$, there
are non-zero integers $m$ and $n$ such that $ux^nu^{-1}=x^m$. Now,
suppose $G^{\ast}$ is the Malcev completion of $G$. We know that
$G^{\ast}$ is a $\mathbb{Q}$-group. By the notations of the above
proof, let $A_i=\{ x_i^{\alpha}:\ \alpha\in \mathbb{Q}\}$, which is
a subgroup of $G^{\ast}$. A similar argument shows that

1- $G\subseteq \bigcup_i A_i$.

2- $A_i\cong \mathbb{Q}$.

3- $A_i\cap A_j=1$, for $i\neq j$.

Since $x_jx_ix_j^{-1}=x_i^{\alpha}$, for some $\alpha\in
\mathbb{Q}$, so $[x_i, x_j]\in A_i$ and similarly it belongs also to
$A_j$. This shows that $x_i$ and $x_j$ are commuting. Therefore, we
have $[x_i^p, x_j^q]=1$, for all $p, q\in \mathbb{Q}$. So, we have
$[A_i, A_j]=1$, which shows that $G$ is abelian.
\end{proof}

\section{Filters of Subgroups}

For standard notions of filter theory, the reader could use
\cite{Bur}. A {\em filter of subgroups} (or a subgroup filter) is
any subset $\mathfrak{F}\subseteq Sub^{\ast}(G)$, with the following
properties:

a) $G\in \mathfrak{F}$.

b) if $A\leq B\leq G$ and $A\in \mathfrak{F}$, then $B\in\mathfrak{F}$.

c) if $A$ and $B\in \mathfrak{F}$, then $A\cap B\in \mathfrak{F}$.

\begin{example}
For any group $G$, the sets $\{G\}$ and $Sub^{\ast}(G)$ are trivial
examples of subgroup filter. For each $x\in G$, there exists the
{\em principal} filter of subgroups
$$
\mathfrak{F}_x=\{ A\in Sub^{\ast}(G): x\in A\}.
$$
If $G$ is an infinite group, then the set
$$
\mathfrak{F}_{cf}=\{ A\leq G: [G:A]<\infty\}
$$
is the {\em cofinite} filter of subgroups.
\end{example}

Like, ordinary filers, any filter of subgroups has {\em finite intersection property} ("fip" for short):
$$
A_1, \ldots, A_n\in \mathfrak{F}\Rightarrow \bigcap_{i=1}^nA_i\neq 1.
$$
On the other hand, any subset $S\subseteq Sub^{\ast}(G)$ with fip is
contained in a unique minimal filter of subgroups (which is the
subgroup filter generated by $S$); if we define
$$
S_{\ast}=\{ \bigcap_{i=1}^nA_i: n\geq 1\ and\ A_i\in S\},
$$
then the subgroup filter generated by $S$ is equal to
$$
\mathfrak{F}_S=\{ B\leq G: A\subseteq B\ for\ some\ A\in S_{\ast}\}.
$$

Let $P(G)$ be the power set of $G$ and $\mathfrak{F}_1\subseteq
P(G)$ be an ordinary filter. Then clearly the set
$\mathfrak{F}=\mathfrak{F}_1\cap Sub^{\ast}(G)$ is a filter of
subgroups. The converse is also true:

\begin{proposition}
Let $\mathfrak{F}\subseteq Sub^{\ast}(G)$ be a subgroup filter. Then there is an ordinary filter $\mathfrak{F}_1\subseteq P(G)$ such that
$$
\mathfrak{F}=\mathfrak{F}_1\cap Sub^{\ast}(G).
$$
\end{proposition}

\begin{proof}
Let
$$
\mathfrak{F}_1=\{ X\subseteq G: A\subseteq X\ for\ some\ A\in \mathfrak{F}\}.
$$
It can be easily verified that $\mathfrak{F}_1$ is an ordinary filter and $\mathfrak{F}=\mathfrak{F}_1\cap Sub^{\ast}(G)$.
\end{proof}

\begin{definition}
An {\em ultra-filter} of subgroups is a subgroup filter
$\mathfrak{F}\subseteq  Sub^{\ast}(G)$ such that for any finite set
of subgroups $A_1, \ldots, A_n$, the condition $\cup_iA_i\in
\mathfrak{F}$ implies $A_i\in \mathfrak{F}$, for some $i$.
\end{definition}

\begin{proposition}
Let $\mathfrak{F}_1\subseteq P(G)$ be an ordinary ultra-filter. Then
$\mathfrak{F}=\mathfrak{F}_1\cap Sub^{\ast}(G)$ is an ultra-filter
of subgroups.
\end{proposition}

\begin{proposition}
Let $\mathfrak{F}$ be a subgroup filter on $G$. Then there exists an ultra-filter of subgroups containing $\mathfrak{F}$.
\end{proposition}

\begin{proof}
We have $\mathfrak{F}=\mathfrak{F}_1\cap Sub^{\ast}(G)$ for some
ordinary filter $\mathfrak{F}_1$. But, there exists an ordinary
ultra-filter $\overline{\mathfrak{F}}_1$ such that
$\mathfrak{F}_1\subseteq \overline{\mathfrak{F}}_1$. Now,
$\overline{\mathfrak{F}}=\overline{\mathfrak{F}}_1\cap
Sub^{\ast}(G)$ is  a subgroup ultra-filter containing
$\mathfrak{F}$.
\end{proof}

\begin{definition}
Let $G$ and $H$ be two groups and $\mathfrak{F}\subseteq Sub^{\ast}(G)$ be a filter of subgroups. Let $f:G\to H$ be a group homomorphism and define
$$
f_{\ast}(\mathfrak{F})=\{ A\leq H: f^{-1}(A)\in \mathfrak{F}\}.
$$
One can see that this is a subgroup filter on $H$ and further, it is a subgroup ultra-filter, if $\mathfrak{F}$ is so.
\end{definition}

The next definition connects two notions of topo-groups and ultra-filters of subgroups.

\begin{definition}
Let $(G, T)$ be a topo-group and $\mathfrak{F}\subseteq
Sub^{\ast}(G)$ be a subgroup ultra-filter. Let $y\in G$. We say that
$\mathfrak{F}$ {\em converges} to $y$, if for all $A\in T$ with
$y\in A$, we have $A\in \mathfrak{F}$. In this situation we write
$\mathfrak{F}\to y$.
\end{definition}
Note that if $(G, T)$ and $(H, S)$ are topo-groups and if $f:G\to H$
is a topomorphism, then for an ultra-filter of subgroup
$\mathfrak{F}$, the condition $\mathfrak{F}\to x$ implies
$f_{\ast}(\mathfrak{F})\to f(x)$, for all $x\in G$. We are now ready
to prove the main theorem of this section:

\begin{theorem}
Let $(G, T)$ be a topo-group. Then $G$ is topo-compact if and only if any ultra-filter of subgroups converges to some point in $G$.
\end{theorem}

\begin{proof}
Let $G$ be topo-compact and $\mathfrak{F}$ be an ultra-filter of
subgroups on $G$. Suppose by contrary that $\mathfrak{F}$ does not
converge to any element of $G$. So, for all $y\in G$, there exists a
topen subgroup $A_y\leq G$ such that $y\in A_y$ and $A_y$ does not
belong to $\mathfrak{F}$. We have $G=\cup_yA_y$, so, $G$ can be
covered by a finite number of $A_y$'s, say $A_{y_1}, \ldots
A_{y_n}$. Hence
$$
A_{y_1}\cup\cdots \cup A_{y_n}=G\in \mathfrak{F},
$$
and therefore $A_{y_i}\in \mathfrak{F}$, for some $i$, since
$\mathfrak{F}$ is ultra-filter. This is a contradiction, so
$\mathfrak{F}$ has at least one convergence point in $G$.

Now, suppose that $G$ is not topo-compact. Hence, there is a
covering $G=\cup_{i\in I}A_i$, where each $A_i$ is topen and $G\neq
\cup_{i\in J}A_i$ for any finite subset $J\subseteq I$. This means
that the set
$$
S=\{ G\setminus A_i: i\in I\}
$$
has "set-fip" (i.e. the intersection of any finite number of its
elements is non-empty) and so there exists an ordinary ultra-filter
$\mathfrak{F}_1$ containing $S$. Let
$\mathfrak{F}=\mathfrak{F}_1\cap Sub^{\ast}(G)$. Then $\mathfrak{F}$
is a subgroup ultra-filter. For any $y\in G$, we have $y\in A_i$ for
some $i$. But as $\mathfrak{F}_1$ is an ultra-filter,  $A_i$ does
not belong to $\mathfrak{F}_1$ and already it does not belong to
$\mathfrak{F}$. This shows that $\mathfrak{F}$ does not converge to
$y$.
\end{proof}

We say that two elements $x$ and $y$ in a group $G$ are {\em
cyclically distinct}, if $\langle x\rangle\cap \langle y\rangle=1$.
So the above theorem says that a topo-system is topo-compact if and
only if any ultra-filter of subgroups have at least one point of
convergence up to the cyclic distinction. A similar assertion holds
for Hausdorff topo-groups.

\begin{theorem}
A topo-group $(G, T)$ is Hausdorff, if and only if any ultra-filter of subgroups in $G$ converges to at most one point (up to cyclic distinction).
\end{theorem}

\begin{proof}
Let $\mathfrak{F}$ be an ultra filter of subgroups in $G$ and $(G,
T)$ be Hausdorff. Let $x$ and $y$ be cyclically distinct but
$\mathfrak{F}\to x$ and in the same time $\mathfrak{F}\to y$. We
know that there are $A, B\in T$ containing $x$ and $y$, respectively
and $A\cap B=1$. But, then we have also $A, B\in \mathfrak{F}$ and
so $1=A\cap B\in \mathfrak{F}$, which is impossible.

Conversely, let $(G, T)$ be not Hausdorff. Then There are cyclically
distinct $x$ and $y$, such that for all topen subgroups $A$ and $B$
with $x\in A$ and $y\in B$, we have $A\cap B\neq 1$. Suppose
$S=\mathfrak{F}_x\cup \mathfrak{F}_y$. Then clearly $S$ has "fip"
and hence there exists an ultra-filter of subgroups, say
$\mathfrak{F}$, such that $S\subseteq \mathfrak{F}$. Now suppose
$A\in T$ and $x\in A$. Then $A\in S\subseteq \mathfrak{F}$.
therefore $\mathfrak{F}\to x$. Similarly we have $\mathfrak{F}\to
y$.
\end{proof}

\section{A Tychonoff Type Theorem}
It seems that many known theorems of topology  have versions in
topo-groups. Once we prove such a theorem, it is possible to
translate it in the language of groups and find an interesting
theorem of group theory. In this section, we prove an analogue of
the compactness theorem of Tychonoff for topo-groups.

\begin{theorem}
Let $\{ (G_i, T_i)\}_{i\in I}$ be a family of topo-compact topo-groups. Then $G=\prod_{i\in I}G_i$ is also topo-compact.
\end{theorem}

\begin{proof}
Let $\mathfrak{F}\subseteq Sub^{\ast}(G)$ be a subgroup
ultra-filter. We prove that it converges to at least one point in
$G$. Note that $(\pi_i)_{\ast}(\mathfrak{F})$ is an ultra-filter of
subgroups in $G_i$ and since $G_i$ is topo-compact, so it converges
to some point $x_i\in G_i$. Suppose $x=(x_i)\in G$. We prove that
$\mathfrak{F}\to x$. Let $A\leq G$ be topen and $x\in A$. We have
$A=\prod_{i\in I}A_i$, with $A_i\in T_i$ and such that almost all
$A_i=G_i$. Since $x_i\in A_i$, so $A_i\in
(\pi_i)_{\ast}(\mathfrak{F})$, i.e. $\pi_i^{-1}(A_i)\in
\mathfrak{F}$. Now, clearly
$$
A=\bigcap_{i\in I}\pi_i^{-1}(A_i),
$$
and the intersection consists of finitely many non-trivial element
of $\mathfrak{F}$. So, it belongs to $\mathfrak{F}$ and this shows
that $\mathfrak{F}\to x$, therefore $G$ is topo-compact.
\end{proof}

{\bf Acknowledgement} The author would like to thank M. H. Jafari
for useful discussion.

\end{document}